\documentclass[reqno]{amsart}
\usepackage{hyperref}
 \usepackage{amsmath} 
 \usepackage{amssymb} 

\begin{document}
\title[ GENERALIZED OF THE ZK EQUATION ]
{THE IVP FOR CERTAIN DISPERSION GENERALIZED \\ OF THE ZK EQUATION IN THE CYLINDER SPACE}

\author[Carolina Albarracin, Guillermo Rodriguez]
{Carolina Albarracin, Guillermo Rodriguez-Blanco}

\address{Carolina Albarracin \newline
Departamento de Matematicas,
 Universidad Nacional de Colombia, 
Carrera 30, calle 45, Bogot\'a, Colombia}
\email{calbarracinh@unal.edu.co}

\address{Guillermo Rodriguez \newline
Departamento de Matematicas,
 Universidad Nacional de Colombia, 
Carrera 30, calle 45, Bogot\'a, Colombia}
\email{grodriguezb@unal.edu.co}

\keywords{Well-posedness, Sobolev spaces, Zakharov-Kutnesov equation
 \hfill\break\indent }

\begin{abstract}
 We establish well-posedness  for the Cauchy problem associated to the dispersion generalized Zakharov-Kutnesov equation in the cylinder. Our main ingredient is a localized Strichartz estimate and an argument of compactness. 
\end{abstract}

\maketitle
\numberwithin{equation}{section}
\newtheorem{theorem}{Theorem}[section]
\newtheorem{lemma}[theorem]{Lemma}
\newtheorem{definition}[theorem]{Definition}
\newtheorem{proposition}[theorem]{Proposition}
\newtheorem{remark}[theorem]{Remark}
\allowdisplaybreaks

\section{Introduction}
We deal with the  well-posedness of the initial value problem (IVP) in the cylinder:
\begin{align}\label{bipe}
	\begin{cases}
		\partial_{t}u-\partial_{x}\left(D_{x}^{1+\alpha}\pm D_{y}^{1+\beta}\right)u+uu_{x}=0 & (x,y)\in\mathbb{R}\times\mathbb{T},\; t\in\mathbb{R}, \\
		u\left(0\right)=\phi & \phi\in H^{s}(\mathbb{R}\times\mathbb{T}),
	\end{cases}
\end{align}
where   $\alpha>-1$, $\beta\geq1$,  $u=u(x,y,t)$ is a real-valued  function,   $\mathbb{T}=\mathbb{R}/(2\pi\mathbb{Z})$ and $D^{\beta}_{y}=(
-\partial_{y}^{2})^{\frac{\beta}{2}}$. The  homogeneous fractional derivative in the variable $y$ is defined via Fourier transform by $\widehat{(D^{\beta}_{y}f)}(m,n):
=|n|^{\beta}\hat{f}(m,n)$ ( analogously is defined $D_{x}^{\alpha}$ for the variable $x$). Moreover,

\begin{align}\label{masa}
	M(u)=\int_{\mathbb{R}\times\mathbb{T}}u^{2}dxdy
\end{align}
and
\begin{align}\label{energia}
	E(u)=\frac{1}{2}\int_{\mathbb{R}\times\mathbb{T}}\left(\left(D^ {\frac{1+\alpha}{2}}_{x}u\right)^{2}\pm\left(D_{y}^{\frac{1+\beta}{2}}u\right)^{2}-\frac{u^{3}}{3}\right)dxdy,
\end{align}
 are conserved by the flow of the equation associated with the IVP
(\ref{bipe}).
The equation in ($\ref{bipe}$) 
with,  $\alpha=\beta=1$ and the positive  sign, is the well-known Zakharov-Kuznetsov (ZK) equation, that is a natural bi-dimensional extension of the well-known Korteweg de Vries (KdV) equation and describes the propagation of nonlinear ion-sonic waves in a magnetized plasma (see \cite{zk} for more  information to this respect). The well-posedness for  the ZK equation in the cylinder $\mathbb{R}\times\mathbb{T}$ was treated by Linares, Pastor and Saut in \cite{linares2010well}, where they adapt the method used by Ionescu and Kenig   \cite{ionescu2009local} to the KP-I equation in the same setting, obtaining local well-posedness in $H^{s}(\mathbb{R}\times\mathbb{T})$ for $s>\frac{3}{2}$. 
Subsequently, this result was improve by Molinet and Pilod in \cite{molinet2015bilinear}
who proved global well-posed in  $H^{1}(\mathbb{R}\times\mathbb{T})$, the main ingredient of the proof is a bilinear Strichartz estimate in the context
of Bourgain’s spaces,  which allows to control the high-low frequency interactions
appearing in the nonlinearity of the equation. 
For $\alpha=0$,   $\beta=1$ and the positive sign, the equation (\ref{bipe}) coincides with  the Benjamin-Ono-Zakharov-Kuznetsov (BOZK) equation that
is a model for thin nano-conductors on a dielectric
substrate in   \cite{latorre2006evolution}(see \cite{esfahani2009instability,esfahani2011ill,cunha2016ivp} for  information of  well-posedness).
However, excluding the above cases, the equation (\ref{bipe}) does not seem to have known
physical relevance. Therefore, the equation serves as a mathematical model to
study the  dispersive effects on the $x$ and $y$ directions.
Other  equations similar to (\ref{bipe}) in  $\mathbb{R}^{2}$
 are the generalized  Benjamin-Ono-Zakharov-Kuznetsov (gBOZK) equation,
 \begin{align}
     \partial_{t} u -D_{x}^{1+\alpha}u_{x}+u_{xyy}=uu_{x}
 \end{align}
and  the fractional Zakharov-Kuznetsov (fZK)
equation, 
\begin{align}\label{shipa}
    \partial_{t}u+\partial_{x}(-\Delta)^{\alpha/2}u=uu_{x}.
\end{align}
The wellposedness for the gBOZK  equation was established by F. Ribaud and S. Vento in \cite{ribaud2016local}  where they proved local  well-posedness for $s>\frac{2}{\alpha}-\frac{3}{4}$ with $0\leq\alpha\leq1$, their  prove consists in an energy method
combined with linear and nonlinear estimates in the short-time Bourgain’s spaces. For  the fZK equation  (\ref{shipa})  R.  Shippa in \cite{schippa2019cauchy}
proves local  well-posedness for  $s>\frac{5}{2}-\alpha$ with $1\leq\alpha\leq2$,   as consequence of  a
short-time bilinear Strichartz estimate, join with  perturbative and energy methods.
Our goal in this work is to improve the local well-posedness in Sobolev space $H^{s}(\mathbb{R}\times\mathbb{T})$, $s>2$ for (\ref{bipe}), which is obtained from parabolic regularization. 
To improve this result we use localized Strichartz's estimates, more accurately we adapt the method used to prove local well-posedness for the Cauchy problem associated to the third-order KP-I and fifth-order KP-I equations on $\mathbb{R}\times\mathbb{T}$  proposed in  \cite{ionescu2009local}. 

Our main results are established below,


\begin{theorem}\label{teoimpo2}
	Let,    $\alpha>-1$,  $\beta\geq1$, $\phi\in H^{s}\left(\mathbb{R}\times\mathbb{T}\right)$ and $s>\frac{6-\alpha}{4}-\frac{1}{2(1+\beta)}+\frac{\lceil \beta \rceil}{2}$,
	where $\lceil\: \rceil$ is the ceiling funtion. Then,  there exist  $T=T\left(\left\Vert \phi\right\Vert _{H^{s}}\right)$ and a unique solution of IVP (\ref{bipe}), in the class
	\begin{align*}
	    C\left(\left[0,T\right];H^{s}\left(\mathbb{R}\times\mathbb{T}\right)\right)\cap L^{1}\left(\left[0,T\right];W^{1,\infty}\left(\mathbb{R}\times\mathbb{T}\right)\right).
	\end{align*}
	 Moreover, the map $\phi\in H^{s}\left(\mathbb{R}\times\mathbb{T}\right)\mapsto u\in C\left(\left[0,T\right];H^{s}\left(\mathbb{R}\times\mathbb{T}\right)\right)$ is continuous.
\end{theorem}
This result improve the local well-posedness obtained via parabolic regularization, for specific values of $\alpha$ and $\beta$ for example:
If $\beta=1$ and $\alpha>-1$,  if $1<\beta\leq2$ and $\alpha>1$, if $2<\beta\leq 3$ and  $\alpha>\frac{10}{3}$, if $3<\beta\leq 4$ and  $\alpha>\frac{11}{2}$. In general, if we aim to maintain $s<2$,  growth of $\beta$  implies  growth of $\alpha$. The restriction on $\beta$  in the Theorem  are necessary to guarantee the
convergence both  of the oscillatory integral in    ($\ref{oscilacio}$) bellow and  of the integral ($\ref{249ec}$) bellow.

The paper is organized as follows. In the following section we prove Localized Strichartz Estimate. Section 3, deals with Preliminary and Key estimates. Lastly, the main results are  proved. First, some notation is necesary:

\textbf{\textit{Notation}}. 
\begin{itemize}
	\item $a\lesssim b$ (resp. $a\gtrsim b$) means that
	there exists a positive constant $c$, such that, $a \leq cb$ (resp. $a\geq cb$). 
	\item $a\sim  b$, when $a\lesssim b$ and $a\gtrsim b$. 
\item $\lceil x\rceil= min\{k\in\mathbb{Z}\hspace{0.2cm}| \hspace{0.2cm} x\leq k\}$
	\item \begin{align*}
		\left\Vert u\right\Vert_{L^{p}X}:=\left(\int_{\mathbb{R}}\left(\left\Vert u(t)\right\Vert_{X}\right)^{p}dt\right)^{\frac{1}{p}}\\
		\left\Vert u\right\Vert_{L^{\infty}X}:=\text{ess sup}_{t\in\mathbb{R}}\left\Vert u(t)\right\Vert_{X},
	\end{align*}
	where $X$ is Banach space,  $u:\mathbb{R} \to X$ is a measurable  function and    $p\in[1,\infty]$.
	\item \begin{align*}
		\left\Vert u\right\Vert_{L^{p}_{I}X}:= \left\Vert \chi_{I}(|t|) u\right\Vert_{L^{p}X},
	\end{align*}
	where $I\subseteq\mathbb{R}$ is interval, $\chi_{I}$ is the characteristic function of $I$ and $u:I \to X$ is a measurable  function. In particular, for $T>0$
	\begin{align*}
		\left\Vert u\right\Vert_{L^{p}_{T}X}:=\left\Vert u\right\Vert_{L^{p}_{[-T,T]}X}.
	\end{align*}
	\item 
	\begin{align*}
		\widehat{f}\left(\xi,n\right)= \int_{\mathbb{R}\times\mathbb{T}}f\left(x,y\right)e^{-ix\xi}e^{-iyn}dxdy  , 
	\end{align*}
	where  $f\in L^{1}\left(\mathbb{R}\times\mathbb{T}\right)$ and  $\left(\xi,n\right)\in\mathbb{R}\times\mathbb{Z}$ . It is the Fourier transform of $f$. Same notation if $f\in\mathcal{S}'$.
	\item  $\mathcal{S}= \{f \in C^\infty (\mathbb{R}\times\mathbb{T})\ |\ ||f||_{\alpha\beta}=\sup\limits_{x\in\mathbb{R}}\big|x^\alpha\partial^\beta_{x} f(x,y)\big|<\infty,\ \forall\alpha,\beta \in \mathbb{N} \}$.
	\item $\mathcal{S}'$ is the distribution space.
	
	\item
	\begin{equation*}
		H^s(\mathbb{R}\times\mathbb{T}) = \{f \in \mathcal{S}' \, : \, \int_{\xi\in\mathbb{R}}\sum_{n\in\mathbb{Z}}(1 + \xi^2+n^2)^s \, |\hat{f}(\xi,n)|^2 \,   <\infty\},
	\end{equation*}
denote the usual Sobolev spaces in $L^2(\mathbb{R}\times\mathbb{T})$  for  $s\in\mathbb{R}$,
	and
	\begin{align*}
		H^{\infty}\left(\mathbb{R}\times\mathbb{T}\right)=\underset{s\geq0}{\cap}H^{s}\left(\mathbb{R}\times\mathbb{T}
		\right).
	\end{align*}
\begin{equation*}
	\left\Vert f\right\Vert _{H^{s}}\sim\left\Vert J_{y}^{s}f\right\Vert _{L^{2}\left(\mathbb{R}\times\mathbb{T}\right)}+\left\Vert J_{x}^{s}f\right\Vert _{L^{2}\left(\mathbb{R}\times\mathbb{T}\right)} 
\end{equation*}
	\item
	For integers $j,k=0,1,\cdots$ we define the operators $Q_{x}^j$ and $Q_{y}^k$ on $H^\infty(\mathbb{R}\times\mathbb{T})$ by
	\begin{equation}
		\begin{cases}
			\widehat{Q_{x}^{0}g}\left(\xi,n\right)=\chi_{\left[0,1\right)}\left(\left|\xi\right|\right)\hat{g}\left(\xi,n\right)\\
			\widehat{Q_{x}^{j}g}\left(\xi,n\right)=\chi_{\left[2^{j-1},2^{j}\right)}\left(\left|\xi\right|\right)\hat{g}\left(\xi,n\right) & if\:j\geq1
		\end{cases}
	\end{equation}
	and
	\begin{equation}
		\begin{cases}
			\widehat{Q_{y}^{0}g}\left(\xi,n\right)=\chi_{\left[0,1\right)}\left(\left|n\right|\right)\hat{g}\left(\xi,n\right)\\
			\widehat{Q_{y}^{k}g}\left(\xi,n\right)=\chi_{\left[2^{k-1},2^{k}\right)}\left(\left|n\right|\right)\hat{g}\left(\xi,n\right) & if\:k\geq1,
		\end{cases}
	\end{equation}
where $\chi_{I}$ is the characteristic function of $I$.
	\item We observe that an equivalence for the Sobolev norms on $\mathbb{R}\times\mathbb{T}$,
	\begin{align}
		\left\Vert g
		\right\Vert^{2}_{L^{2}(\mathbb{R}\times\mathbb{T})}\sim\sum_{k,j\geq 0}\left\Vert Q_{x}^{j}Q_{y}^{k}g\right\Vert^{2}_{L^{2}(\mathbb{R}\times\mathbb{T})}
	\end{align}
and
		\begin{align}
		\left\Vert J^{s}_{x} g
		\right\Vert^{2}_{L^{2}(\mathbb{R}\times\mathbb{T})} +\left\Vert J^{s}_{y} g
		\right\Vert^{2}_{L^{2}(\mathbb{R}\times\mathbb{T})} &\sim\sum_{k\geq0,j\geq 1}(2^{j})^{2s}\left\Vert Q_{x}^{j}Q_{y}^{k}g\right\Vert^{2}_{L^{2}(\mathbb{R}\times\mathbb{T})}\\&\nonumber+	\sum_{k\geq1,j\geq 0}(2^{k})^{2s}\left\Vert Q_{x}^{j}Q_{y}^{k}g\right\Vert^{2}_{L^{2}(\mathbb{R}\times\mathbb{T})}\\&\nonumber +\sum_{k,j\geq 0}\left\Vert Q_{x}^{j}Q_{y}^{k}g\right\Vert^{2}_{L^{2}(\mathbb{R}\times\mathbb{T})}
	\end{align}
\end{itemize}

\section{Localized Strichartz Estimates}
In this section, we prove Strichartz estimates, localized in frequency and time. 

Let
$\left\{W_{0}\left(t\right)\right\}_{t}$ be the group associated with the linear part of the equation (\ref{bipe}), that is, 

\begin{align*}
	W_{0}\left(t\right)\phi&=\int_{\xi\in\mathbb{R}}\sum_{n\in\mathbb{Z}}\widehat{\phi}\left(\boldsymbol{\xi}\right)e^{i\left(\boldsymbol{\xi}\textbf{x}+\xi\left(\left|\xi\right|^{1+\alpha}\pm \left|n\right|^{1+\beta}\right)t\right)},
\end{align*}

where  $\boldsymbol{\xi}=\left(\xi,n\right)$, $\textbf{x}=\left(x,y\right)$ and $\boldsymbol{\xi}\textbf{x}=\xi x+ny$.

\begin{theorem}
	\label{striper2}
Let  $\alpha>-1$, $\beta\geq1$ and $\phi\in \mathcal{S}\left(\mathbb{R}\times\mathbb{T}\right)$  
	, then for any $\epsilon>0$, 
	\begin{align}\label{striper2.1}
		\left\Vert \mathbb{W}_{0}\left(.\right)Q_{y}^{k}Q_{x}^{j}\phi\right\Vert _{L_{2^{-(\lceil\beta\rceil k+j)}}^{2}L_{xy}^{\infty}\left(\mathbb{R}\times\mathbb{T}\right)}\lesssim_{\epsilon}2^{-\left(\frac{\alpha}{4}+\frac{1}{2(1+\beta)}+\epsilon\right)j}\left\Vert Q_{y}^{k}Q_{x}^{j}\phi\right\Vert _{L_{xy}^{2}\left(\mathbb{R}\times\mathbb{T}\right)}.
	\end{align}
	
\end{theorem}

First we recall the following lemmas;
\begin{lemma}
[Poisson Summation Formula]\label{sumpoisson}
If $f$, $\hat{f}$ are in $L^{1}(\mathbb{R}^{n})$ and  satisfy the condition  
	\begin{align*}
		\left|f\left(x\right)\right|+\left|\hat{f}\left(x\right)\right|\leq C\left(1+\left|x\right|\right)^{-n-\delta}   \end{align*}
	for some $C$, $\delta>0$. Then,    
	$f$ and  $\hat{f}$ are continuous and
	\begin{align*}
	\sum_{m\in\mathbb{Z}^{n}}\widehat{f}\left( m\right)=\sum_{k\in\mathbb{Z}^{n}}f\left(2\pi k\right)   
	\end{align*}
	
\end{lemma}
\begin{proof}
	See \cite{grafakos2008classical}  Theorem 3.1.17
\end{proof}

\begin{lemma}[Van der Corput]\label{vander}
	Let $p\geq2$, $I
	=[a,b]$,  $\varphi\in C^{p}\left(I\right)$ is a real function  such that $\left|\varphi^{\left(p\right)}\left(x\right)\right|\geq\lambda>0$,
	$\psi\in L^{\infty}\left(I\right)$ and $\psi'\in L^{1}\left(I\right)$,
	then,
	
	\begin{align*}
	  \left|\int_{I}e^{i\varphi\left(x\right)}\psi\left(x\right)dx\right|\leq C_{p}\lambda^{\frac{1}{p}}\left(\left\Vert \psi\right\Vert _{L^{\infty}}+\left\Vert \psi'\right\Vert _{L^{1}}\right).  
	\end{align*}

\end{lemma}
\begin{proof}
	See \cite{stein1993harmonic} Chapter 8.
\end{proof}
\begin{lemma}
\label{io}
		Let $f\in C_{0}^{\infty}\left(\left[a,b\right]\right)$ and  $\tilde{\phi}'\left(x\right)\neq0$
		for any $x\in\left[a,b\right]$.  Then, 
		\begin{equation*}
			I\left(\lambda\right)=\int_{a}^{b}e^{i\lambda\tilde{\phi}\left(x\right)}f\left(x\right)dx=O\left(\lambda^{-k}\right),   
		\end{equation*}
		
		as $\lambda\rightarrow\infty$, for any $k\in\mathbb{Z}^{+}$
	
\end{lemma}
\begin{proof}
	See \cite{linpon} Chapter 1.
\end{proof}

\begin{proof}[\textbf{Proof of Theorem 2.1}]
	Let $\psi_{1}:\mathbb{R}\rightarrow\left[0,1\right]$ denote a smooth even function, supported in  $\left\{ r\mid\left|r\right|\in\left[\frac{1}{4},4\right]\right\} $ and    $\psi_{1}\equiv1$
	in  $\left\{ r\mid\left|r\right|\in\left[\frac{1}{2},2\right]\right\} $,   $\psi_{0}:\mathbb{R}\rightarrow\left[0,1\right]$ denote a smooth even function,  supported in $[-2,2]$ and    $\psi_{0}\equiv1$
	in  $\left[-1,1\right] $, and  
	$a\left(\boldsymbol{\xi}\right)=\widehat{\left(Q_{y}^{k}Q_{x}^{j}\phi\right)}\left(\boldsymbol{\xi}\right)$. As,   $\psi_{1}\left(\frac{\xi}{2^{j}}\right)\psi_{0}\left(\frac{n}{2^{k}}\right)=1$, in  $\left[-2^{j+1},-2^{j-1}\right]\times\left[-2^{k},0\right]\cup\left[2^{j-1},2^{j+1}\right]\times\left[0,2^{k}\right]$, then,
	 \begin{equation*}
	 	\mathbb{W}_{0}\left(t\right)Q_{y}^{k}Q_{x}^{j}\phi\left(\textbf{x}\right)=\int_{\mathbb{R}}\sum_{n\in\mathbb{Z}} a\left(\boldsymbol{\xi}\right)\psi_{1}\left(\frac{\xi}{2^{j}}\right)\psi_{0}\left(\frac{n}{2^{k}}\right)e^{i\left(\boldsymbol{\xi}\textbf{x}+F\left(\boldsymbol{\xi}\right)t\right)}d\xi,
	 \end{equation*}
 where,
 \begin{align*}
 	F\left(\boldsymbol{\xi}\right)=\xi\left|\xi\right|^{1+\alpha}\pm\xi \left|n\right|^{1+\beta}.
 \end{align*}
 Using the same argument as in the proof of Theorem 9.3.2 of \cite{ionescu2009local}, it is enough to prove that,
	\begin{align}\label{des1}
		\left|\int_{\xi\in\mathbb{R}}\sum_{n\in\mathbb{Z}}\psi_{1}^{2}\left(\frac{\xi}{2^{j}}\right)e^{i\left(\xi x+t\xi\left|\xi\right|^{1+\alpha}\right)}\psi_{0}^{2}\left(\frac{n}{2^{k}}\right)e^{i\left(ny\pm t\xi\left|n\right|^{1+\beta}\right)}d\xi\right|\lesssim2^{l-\left(\frac{\alpha}{2}+\frac{1}{1+\beta}\right)j},   
	\end{align}
	for any 
	$x\in\mathbb{R},\;y\in\left[0,2\pi\right)$ and   $\left|t\right|\in\left[2^{-l},2^{-l}2\right]$ and  $l\geq \lceil \beta\rceil k+j$. The  inequality (\ref{des1}), in the case $j=0$ is immediate. For $j>0$,   the Poisson summation formula  Lemma \ref{sumpoisson},  transform the expression inside module on the left-hand side  of (\ref{des1}) to,

	\begin{align*}
		\int_{\xi\in\mathbb{R}}\psi_{1}^{2}\left(\frac{\xi}{2^{j}}\right)e^{i\left(\xi x+t\xi\left|\xi\right|^{1+\alpha}\right)}\underbrace{\sum_{\nu\in\mathbb{Z}}\int_{\mathbb{R}}\psi_{0}^{2}\left(\frac{\eta}{2^{k}}\right)e^{i\left(\left(y-2\pi\nu\right)\eta\pm t\xi\left|\eta\right|^{1+\beta}\right)}d\eta}_{A(\xi,t,y)}.   
	\end{align*}
	
The  integral in $A(\xi,t,y)$, is transformed using integration by parts. That is,

	\begin{align}\label{249ec}
		&\int_{\mathbb{R}}\psi_{0}^{2}\left(\frac{\eta}{2^{k}}\right)e^{i\left(\left(y-2\pi\nu\right)\eta\pm t\xi \left|\eta\right|^{1+\beta}\right)}d\eta\\&\nonumber
		=\int_{\mathbb{\left|\eta\right|\leq}2^{k+1}}\frac{\psi_{0}^{2}\left(\frac{\eta}{2^{k}}\right)}{i\left(\left(y-2\pi\nu\right)\pm(1+\beta)sgn(\eta)\left|\eta\right|^{\beta} \xi t\right)}\frac{d}{d\eta}\left(e^{i\left(\left(y-2\pi\nu\right)\eta\pm t\xi \left|\eta\right|^{1+\beta} \right)}\right)d\eta\\&\nonumber
		=i\int_{\mathbb{\left|\eta\right|\leq}2^{k+1}}\left(\frac{2\cdot2^{-k}\psi_{0}\psi_{0}^{'}}{\left(y-2\pi\nu\right)\pm (1+\beta)sgn(\eta)\left|\eta\right|^{\beta}\xi t}\pm\frac{\psi_{0}^{2}(1+\beta)\beta 
		\left|\eta\right|^{\beta-1}\xi t}{\left(\left(y-2\pi\nu\right)\pm(1+\beta)sgn(\eta)\left|\eta\right|^{\beta} \xi t\right)^{2}}\right) \\&\nonumber\hspace{2cm}
		e^{i\left(\left(y-2\pi\nu\right)\eta+ t\xi\left|\eta\right|^{1+\beta}\right)}d\eta.
	\end{align}
	
For the second term of  the  right-hand  side  of (\ref{249ec}) in the previous integral, we have that,
\begin{align*}
    (1+\beta)\beta\left|\eta\right|^{\beta-1}\xi t\leq 2\beta^{2} 2^{k(\beta-1)}2^{j}2^{-(j+\lceil\beta\rceil k)}\lesssim_{\beta}2^{-k},
\end{align*}
	 \begin{align*}
	   (1+\beta)\left|\eta\right|^{\beta}\xi
	t\lesssim_{\beta} 2^{k\beta}2^{j}2^{-(\lceil\beta\rceil k+j)}\lesssim_{\beta}1,
	 \end{align*}
	$\:y\in\left[0,2\pi\right)$
	and $\left|\nu\right|>100$, then $\left|\left(y-2\pi\nu\right)\pm (1+\beta)sgn(\eta)\left|\eta\right|^{\beta} \xi t\right|^{2}\sim \left|\nu\right|^{2}$. We get,
 \begin{align}\label{integralacot}
	\left|\int_{\mathbb{\left|\eta\right|\leq}2^{k+1}}\frac{\psi_{0}^{2}(1+\beta)\beta\left|\eta\right|^{\beta-1} \xi t}{\left(\left(y-2\pi\nu\right)\pm (1+\beta)sgn(\eta)\left|\eta\right|^{\beta} \xi t\right)^{2}}
	e^{i\left(\left(y-2\pi\nu\right)\eta\pm t\xi\left|\eta\right|^{1+\beta}\right)}d\eta\right|\lesssim_{\beta}2^{k}\frac{2^{-k}}{\left|\nu\right|^{2}}\lesssim_{\beta}\frac{1}{\left|\nu\right|^{2}} .  
\end{align}
	
	Following similar consideration for  the  first  term  of  the  right-hand  side  of (\ref{249ec}). We can conclude that, if $\left|\nu\right|>100$, 
	\begin{align*}
	    \left|\int_{\mathbb{R}}\psi_{0}^{2}\left(\frac{\eta}{2^{k}}\right)e^{i\left(\left(y-2\pi\nu\right)\eta\pm t \xi \left|\eta\right|^{1+\beta}\right)}d\eta\right|\leq\frac{C}{\left|\nu\right|^{3}}+\frac{C}{\left|\nu\right|^{2}}\leq\frac{C}{\left|\nu\right|^{2}}.
	\end{align*}
	Then,
	\begin{align}\label{ecu26}
		A(\xi,t,y)&\nonumber=\sum_{\nu\in \mathbb{Z}}\int_{\mathbb{R}}\psi_{0}^{2}\left(\frac{\eta}{2^{k}}\right)e^{i\left(\left(y-2\pi\nu\right)\eta\pm t\xi\left|\eta\right|^{1+\beta}\right)}d\eta\\&=\sum_{\left|\nu\right|\leq100}\int_{\mathbb{R}}\psi_{0}^{2}\left(\frac{\eta}{2^{k}}\right)e^{i\left(\left(y-2\pi\nu\right)\eta\pm t\xi\left|\eta\right|^{1+\beta}\right)}d\eta+O\left(1\right).  
	\end{align}
	
	The term $O(1)$ 
	 in (\ref{des1}) imply that only we need estimate one integral, which is domined by the right-hand side of (\ref{des1}). Thus, we replace the sum over $n$ with a sum of C integrals, modulo aceptable error, then (\ref{des1}) reduces to the proof of
	 ,

\begin{align}\label{intint2}
	\left|\int_{\mathbb{R}}\psi_{1}^{2}\left(\frac{\xi}{2^{j}}\right)e^{i\left(\xi x+\xi\left|\xi\right|^{1+\alpha}t\right)}\int_{\mathbb{R}}\psi_{0}^{2}\left(\frac{\eta}{2^{k}}\right)e^{i\left(\eta y'\pm\xi t\left|\eta\right|^{1+\beta}\right)}d\eta d\xi\right|\lesssim2^{l-\left(\frac{\alpha}{2}+\frac{1}{1+\beta}\right)}j  
\end{align}
  $\xi\sim2^{j}$,   $\left|t\right| \sim2^{-l}$ and  $\eta\sim2^k$. We write,
	\begin{align}\label{oscilacio}
		\int_{\mathbb{R}}\psi_{0}^{2}\left(\frac{\eta}{2^{k}}\right)e^{i\left(y'\eta \pm t\xi\left|\eta\right|^{1+\beta}\right)}d\eta&=\nonumber\int_{\mathbb{R}}2^{k}\widehat{\psi_{0}^{2}}\left(2^{k}\eta_{1}\right)\widehat{e^{i\left(y'\eta \pm t\xi\left|\eta\right|^{1+\beta}\right)}}(\eta_1)d\eta_{1}\\&=\frac{C}{\left(t\xi\right)^{\frac{1}{1+\beta}}}\int_{\mathbb{R}}\widehat{\psi_{0}^{2}}\left(\lambda\right)H\left(\frac{y'-\frac{\lambda}{2^{k}}}{(t\xi)^{\frac{1}{1+\beta}}}\right)d\lambda
	\end{align}
 $H(\eta_1)=\int_{\mathbb{R}}e^{i(\eta_1\eta\pm \left|\eta\right|^{1+\beta})}d\eta$, $H\in L^{\infty}(\mathbb{R})$ if $\beta\geq1$ (see \cite{kenig1991oscillatory,sidi1986long}). Then, substituting (\ref{oscilacio}) into (\ref{intint2}), it suffices to prove that

\begin{equation}\label{rem2}
		\left|\int_{\mathbb{R}}\frac{1}{\left(\xi t\right)^{\frac{1}{1+\beta}}}\psi_{1}^{2}\left(\frac{\xi}{2^{j}}\right)e^{i\left(\xi x+\xi\left|\xi\right|^{1+\alpha}t\right)} d\xi\right|\lesssim2^{l-\left(\frac{\alpha}{2}+\frac{1}{1+\beta}\right)j}  \end{equation}

where $\left|t\right| \sim2^{-l}$ and $\xi\sim2^{j}$. The Van-der Corput lemma implies that, 
\begin{align}\label{conclusion}
	\left|\int_{\mathbb{R}}\psi_{1}^{2}\left(\frac{\xi}{2^{j}}\right)e^{i\left(x\xi+\xi\left|\xi\right|^{1+\alpha}t\right)}d\xi\right|\leq2^{(l-\alpha j)/2}   \end{align}
	and thus,
\begin{equation}\label{rem3}
\left|\int_{\mathbb{R}}\psi_{1}^{2}\left(\frac{\xi}{2^{j}}\right)e^{i\left(\xi x+\xi\left|\xi\right|^{1+\alpha}t\right)} d\xi\right|\lesssim2^{\frac{\beta}{1+\beta}l-\frac{\alpha}{2}j}  \end{equation}
because  $\frac{\beta}{1+\beta}\geq\frac{1}{2}$ if $\beta\geq1$, then we obtain (\ref{rem2}).

\end{proof}
\section{Preliminary and Key estimates}
As a consequence of the Strichartz inequality  Theorem \ref{striper2}, we obtain:
\begin{lemma}\label{unot}
	Let $\alpha>-1$,  $\beta\geq1$, $u\in C\left(\left[0,T\right];H^{\infty}\left(\mathbb{R}\times\mathbb{T}\right)\right)\cap C^{1}\left(\left[0,T\right];H^{\infty}\left(\mathbb{R}\times\mathbb{T}\right)\right)$ and  
	$f\in C\left(\left[0,T\right];H^{\infty}\left(\mathbb{R}\times\mathbb{T}\right)\right)$,
	$T\in\left[0,1\right]$ such that
	\[
	\partial_{t}u-\partial_{x}\left(D_{x}^{1+\alpha}\pm D_{y}^{1+\beta}\right)u=\partial_{x}f.
	\]
	Then, 	
	\begin{equation}\label{eq31}
		\left\Vert u\right\Vert _{L_{T}^{1}L_{xy}^{\infty}\left(\mathbb{R}\times\mathbb{T}\right)}\lesssim_{s_{1}s_{2}}T^{\frac{1}{2}}\left(\left\Vert J_{x}^{s_{1}}J_{y}^{s_{2}}u\right\Vert _{L_{T}^{\infty}L_{xy}^{2}\left(\mathbb{R}\times\mathbb{T}\right)}+\left\Vert J_{x}^{s_{1}}f\right\Vert _{L_{T}^{1}L_{xy}^{2}\left(\mathbb{R}\times\mathbb{T}\right)}\right),  
	\end{equation}
	for any $s_{1}>\frac{2-\alpha}{4}-\frac{1}{2(1+\beta)}$ and $s_{2}>\frac{\lceil\beta\rceil}{2}$ 
\end{lemma}

\begin{proof}[\textbf{Proof of Lemma 3.1}]
	We partition the interval $\left[0,T\right]$ into  $2^{\lceil \beta\rceil k+j}$ equal intervals
	of length $T2^{-\left(\lceil \beta\rceil k+j\right)}$, denote by $\left[a_{k,m},a_{k,\left(m+1\right)}\right)$,  
	$m=0,1,2\cdots2^{j+ \lceil \beta\rceil k}$. We observe that, Cauchy-Schwarz inequality implies that, 
	\begin{align}\label{246}
	\left\Vert Q_{y}^{k}Q_{x}^{j}u\right\Vert _{L_{T}^{1}L_{xy}^{\infty}}&\leq\sum_{m=1}^{2^{\lceil \beta\rceil 
	k+j}}\left\Vert \chi_{\left[a_{k,m},a_{k,\left(m+1\right)}\right)}\left(t\right)Q_{y}^{k}Q_{x}^{j}u\right\Vert _{L_{T}^{1}L_{xy}^{\infty}}
	\\\nonumber
	&\lesssim\left(T2^{-(\lceil\beta\rceil k+j)}\right)^{\frac{1}{2}}\sum_{m=1}^{2^{\lceil \beta\rceil 
	k+j}}\left\Vert \chi_{\left[a_{k,m},a_{k,\left(m+1\right)}\right)}\left(t\right)Q_{y}^{k}Q_{x}^{j}u\right\Vert _{L_{T}^{2}L_{xy}^{\infty}}.
	\end{align}

	Duhamel's formula, 
	\begin{equation}\label{DUHAMEL1}
		u\left(t\right)=W_{0}^{\alpha}\left(t-a_{k,m}\right)\left(u\left(a_{k,m}\right)\right)+\int_{a_{k,m}}^{t}W_{0}^{\alpha}\left(t-s\right)\left(\partial_{x}f\left(s\right)\right)ds,  
	\end{equation}
	for $t\in\left[a_{k,m},a_{k,\left(m+1\right)}\right]$, join with  (\ref{246})
  and   Theorem \ref{striper2}, imply that,
		\begin{align*}
	\left\Vert Q_{y}^{k}Q_{x}^{j}u\right\Vert _{L_{T}^{1}L_{xy}^{\infty}}&\leq2^{-\epsilon\frac{(j+k)}{2}}
		T^{\frac{1}{2}}\left(2^{-(\lceil\beta\rceil k+j)}\sum_{m=1}^{2^{\lceil\beta\rceil k+j}}\left(\left\Vert J_{x}^{s_{1}}J_{y}^{s_{2}}u\left(a_{k,m}\right)\right\Vert _{L_{xy}^{2}}\right)\right.\\&\hspace{4cm}\left.+\left\Vert J_{x}^{s_{1}}f\right\Vert _{L_{T}^{1}L_{xy}^{2}}\right)
	\end{align*}

\end{proof}

To obtain the energy estimate, we need the cylinder version of the Kato -Ponce commutador,
\begin{proposition}
	\label{CONMUTADORBIPE}
	Let $s\geq1$ and $f,g\in H^{\infty}\left(\mathbb{R}\times\mathbb{T}\right)$.
	Then, 
	\begin{equation*}
		\left\Vert J_{y}^{s}\left(fg\right)-fJ_{y}^{s}g\right\Vert _{L^{2}}\leq C_{s} \left\Vert J_{y}^{s}f\right\Vert _{L^{2}}\left\Vert g\right\Vert _{L^{\infty}}+\left(\left\Vert f\right\Vert _{L^{\infty}}+\left\Vert \partial_{y} f\right\Vert _{L^{\infty}}\right)\left\Vert J_{y}^{s-1}g\right\Vert _{L^{2}},   
	\end{equation*}
	$C_{s}$ is a constant.
\end{proposition}
\begin{proof}
	By Lemma 9.A.1 in \cite{ionescu2009local}, 
	\begin{align*}
		\left\Vert J_{y}^{s}\left(fg\right)-fJ_{y}^{s}g\right\Vert _{L^{2}}&=\left\Vert\left\Vert J_{y}^{s}\left(fg\right)-fJ_{y}^{s}g\right\Vert _{L_{y}^{2}}\right\Vert _{L_{x}^{2}}\\
		&\leq C_{s} \left\Vert\left\Vert J_{y}^{s}f\right\Vert _{L_{y}^{2}}\left\Vert g\right\Vert _{L_{y}^{\infty}}+\left(\left\Vert f\right\Vert _{L_{y}^{\infty}}+\left\Vert \partial_{y} f\right\Vert _{L_{y}^{\infty}}\right)\left\Vert J_{y}^{s-1}g\right\Vert _{L_{y}^{2}}\right\Vert _{L_{x}^{2}}
		\\
		&\leq C_{s} \left\Vert J_{y}^{s}f\right\Vert _{L^{2}}\left\Vert g\right\Vert _{L^{\infty}}+\left(\left\Vert f\right\Vert _{L^{\infty}}+\left\Vert \partial_{y} f\right\Vert _{L^{\infty}}\right)\left\Vert J_{y}^{s-1}g\right\Vert _{L^{2}}
	\end{align*}
\end{proof}
\begin{lemma}[Energy estimate]\label{supper} Let $\alpha>-1$, $\beta\geq1$ and $u$ solution of IVP (\ref{bipe}) with $\phi\in H^{\infty}\left(\mathbb{R}\times\mathbb{T}\right)$.
 Then,
	\begin{equation*}
		\underset{0<t<T}{\sup}\left\Vert u\right\Vert _{H^{s}}\lesssim e^{\left(\left\Vert u\right\Vert _{L_{T}^{1}L_{xy}^{\infty}}+\left\Vert \partial_{x}u\right\Vert _{L_{T}^{1}L_{xy}^{\infty}}+\left\Vert \partial_{y}u\right\Vert _{L_{T}^{1}L_{xy}^{\infty}}\right)}\left\Vert \phi\right\Vert _{H^{s}},
	\end{equation*}
	 for any $s\geq1$,  $T\in\left[0,1\right]$.
\end{lemma}
\begin{proof}We apply $J_{y}^{s}$ to equation in (\ref{bipe}),   multiply by $J_{y}^{s}u$ and integrate, to obtain
	\begin{equation*}
	\int_{\mathbb{R}\times\mathbb{T}} J_{y}^{s}u_{t}J_{y}^{s}udxdy-\int_{\mathbb{R}\times\mathbb{T}} J_{y}^{s}\partial_{x}D_{x}^{1+\alpha}uJ_{y}^{s}udxdy\pm\int_{\mathbb{R}\times\mathbb{T}} J_{y}^{s}\partial_{x}D_{y}^{1+\beta}uJ_{y}^{s}udxdy
	\end{equation*}
	\begin{align*}
		+\int_{\mathbb{R}\times\mathbb{T}} J_{y}^{s}uu_{x}J_{y}^{s}udxdy=0,
	\end{align*}
	
	that,  integrating by parts and applying the  Proposition \ref{CONMUTADORBIPE}, transform this equality in
	\begin{align*}
		\frac{1}{2}\frac{d}{dt}\left\Vert J_{y}^{s}u\right\Vert _{L_{xy}^{2}}^{2}&\lesssim\left\Vert \partial_{x}u\right\Vert _{L_{xy}^{\infty}}\left\Vert J_{y}^{s}u\right\Vert _{L_{xy}^{2}}+\left\Vert J_{y}^{s}u\right\Vert _{L_{xy}^{2}}\left\Vert [J_{y}^{s},u]\partial_{x}u\right\Vert _{L_{xy}^{2}}\\&
		\lesssim\left\Vert \partial_{x}u\right\Vert _{L_{xy}^{\infty}}\left\Vert J_{y}^{s}u\right\Vert _{L_{xy}^{2}}
		\\&+\left\Vert \partial_{x}u\right\Vert _{L_{xy}^{\infty}}\left\Vert J_{y}^{s}u\right\Vert _{L_{xy}^{2}}+\left(\left\Vert u\right\Vert _{L_{xy}^{\infty}}+\left\Vert \partial_{y} u\right\Vert _{L_{xy}^{\infty}}\right)\left\Vert
		J_{y}^{s-1}\partial_{x}u\right\Vert _{L_{xy}^{2}}
	\end{align*}
A similar argument shows that
\begin{align*}
		\frac{1}{2}\frac{d}{dt}\left\Vert J_{x}^{s}u\right\Vert _{L_{xy}^{2}}^{2}&\lesssim\left\Vert \partial_{x}u\right\Vert _{L_{xy}^{\infty}}\left\Vert J_{x}^{s}u\right\Vert _{L_{xy}^{2}}+\left\Vert J_{x}^{s}u\right\Vert _{L_{xy}^{2}}\left\Vert [J_{x}^{s},u]\partial_{x}u\right\Vert _{L_{xy}^{2}}.
\end{align*}
Gronwall's inequality, implies the result.
\end{proof}
Now, we remember some facts that we need for the next result.
\begin{theorem}\label{ruleLeib}
	 [Leibniz rule] Let $1<p<\infty$, $0<s<1$ and $f,g\in H^{\infty}(\mathbb{R}\times\mathbb{T})
	 $ Then,
	\begin{equation*}
		\left\Vert J_{x}^{s}\left(fg\right)\right\Vert _{L^{p}}\leq C \left(\left\Vert J_{x}^{s}f\right\Vert _{L^{\infty}}\left\Vert g\right\Vert _{L^{p}}+\left\Vert f\right\Vert _{L^{\infty}}\left\Vert J_{x}^{s}g\right\Vert _{L^{p}}\right),  
	\end{equation*}

\end{theorem}
\begin{proof}
	By Kenig, Ponce and Vega in \cite{kenig1993well} we get,
	\begin{align*}
		\left\Vert J_{x}^{s}\left(fg\right)\right\Vert _{L^{p}}&\leq\left\Vert\left\Vert J_{x}^{s}\left(fg\right)\right\Vert _{L_{x}^{p}}\right\Vert _{L_{y}^{p}}\\&\leq C\left\Vert \left(\left\Vert J_{x}^{s}f\right\Vert _{L_{x}^{\infty}}\left\Vert g\right\Vert _{L_{x}^{p}}+\left\Vert f\right\Vert _{L_{x}^{\infty}}\left\Vert J_{x}^{s}g\right\Vert _{L_{x}^{p}}\right)\right\Vert _{L_{y}^{p}} 
		\\&\leq
		C\left(\left\Vert J_{x}^{s}f\right\Vert _{L^{\infty}}\left\Vert g\right\Vert _{L^{p}}+\left\Vert f\right\Vert _{L^{\infty}}\left\Vert J_{x}^{s}g\right\Vert _{L^{p}}\right)
	\end{align*}
\end{proof}
\begin{lemma}
	For every $0<s<1$ there exists a constant C such that for every $u\in L_{xy}^{\infty}$ satisfying $\partial_{x}u\in L_{xy}^{\infty}$, one has
	\begin{align*}
			\left\Vert J_{x}^{s}u\right\Vert _{L_{x}^{\infty}}\leq C\left(\left\Vert u\right\Vert _{L_{x}^{\infty}}+\left\Vert \partial_{x}u\right\Vert _{L_{x}^{\infty}}\right)
	\end{align*}
\end{lemma}
\begin{proof}
	See Molinet, Saut and Tzvetkov in \cite{molinet2007global}
\end{proof}
As a consequence of Lemma \ref{unot} , we obtain:
\begin{proposition}\label{constante}
	Let $\alpha>-1$, $\beta\geq1$, $u$  be a solution del IVP (\ref{bipe}) with  $\phi\in H^{\infty}\left(\mathbb{R}\times\mathbb{T}\right)$. Then,  for any  $s>\frac{6-\alpha}{4}-\frac{1}{2(1+\beta)}+\frac{\lceil\beta\rceil}{2}$,  there exists  $T=T(\left\Vert \phi\right\Vert _{s},s)$ and  a constant $C_{T}(\left\Vert \phi\right\Vert _{s},s)$ such that,
	\begin{align}\label{constantestyc}
		g\left(T\right):=\int_{0}^{T}\left(\left\Vert u\right\Vert _{L_{xy}^{\infty}\left(\mathbb{R}\times\mathbb{T}\right)}+\left\Vert u_{x}\right\Vert _{L_{xy}^{\infty}\left(\mathbb{R}\times\mathbb{T}\right)}+\left\Vert u_{y}\right\Vert _{L_{xy}^{\infty}\left(\mathbb{R}\times\mathbb{T}\right)}\right)dt'\leq C_{T}
	\end{align}
\end{proposition}

\begin{proof}[\textbf{Proof of Proposition 3.7}]
	We apply Lemma \ref{unot}  with $s_{1}>\frac{2-\alpha}{4}-\frac{1}{2(1+\beta)}$  and  $s_2>\frac{\lceil \beta \rceil}{2}$  in  $u$,
	$\partial_xu$,  $\partial_{y}u$ and respectively by $f=\frac{1}{2}u^{2},\frac{1}{2}\partial_xu^{2},\frac{1}{2}\partial_yu^{2}$. 
	We get,
	\begin{align*}
		&\left\Vert u\right\Vert _{L_{T}^{1}L_{xy}^{\infty}}+\left\Vert u_{x}\right\Vert _{L_{T}^{1}L_{xy}^{\infty}}+\left\Vert u_{y}\right\Vert _{L_{T}^{1}L_{xy}^{\infty}}\\&
		\lesssim_{s}T^{\frac{1}{2}}\left(\left\Vert J_{x}^{s_{1}}J_{y}^{s_{2}}u\right\Vert _{L_{T}^{\infty}L_{xy}^{2}}+\left\Vert J_{x}^{s_{1}}J_{y}^{s_{2}}\partial_{x}u\right\Vert _{L_{T}^{\infty}L_{xy}^{2}}+\left\Vert J_{x}^{s_{1}}J_{y}^{s_{2}}\partial_{y}u\right\Vert _{L_{T}^{\infty}L_{xy}^{2}}\right.\\&+ \left.\left\Vert J_{x}^{s_{1}}u^{2}\right\Vert _{L_{T}^{1}L_{xy}^{2}}+\left\Vert J_{x}^{s_{1}} \partial_{x}\left(u^{2}\right)\right\Vert _{L_{T}^{1}L_{xy}^{2}}+\left\Vert J_{x}^{s_{1}} \partial_{y}\left(u^{2}\right)\right\Vert _{L_{T}^{1}L_{xy}^{2}}\right)  
	\end{align*}
The first three terms on the right-hand side the above inequality are estimate respectively by Young's inequality, since $\frac{s_2}{s_1+s_2}+\frac{s_1}{s_1+s_2}=1$, $\frac{s_2}{s_1+s_2+1}+\frac{s_1+1}{s_1+s_2+1}=1$ and  $\frac{s_1}{s_1+s_2+1}+\frac{s_2+1}{s_1+s_2+1}=1$.
The last three can be estimated   by applying  Leibniz's product rule  Theorem \ref{ruleLeib}.	Then, we obtain the inequality,
	\begin{align*}
		g\left(T\right)\lesssim\left\Vert \phi\right\Vert _{s}e^{g\left(T\right)}\left(1+g\left(T\right)\right).
	\end{align*}
 An argument of continuity complete the proof, if $T\leq T_{0}C_{T}(\left\Vert \phi\right\Vert _{s},s)$ is small enough.
\end{proof}

In this point, we can use standard compactness arguments as in Kenig \cite{kenig2015well} for proving Theorem \ref{teoimpo2}.  
\section{Proof of Theorem 1.1}
Let $s>\frac{6-\alpha}{4}-\frac{1}{2(1+\beta)}+\frac{\lceil\beta\rceil}{2}$,  $\phi\;\in H^{s}\left(\mathbb{R}\times\mathbb{T}\right)$, $\phi_{\gamma}\;\in H^{\infty}\left(\mathbb{R}\times\mathbb{T}\right)$,  such that $\underset{\gamma\rightarrow\infty}{\lim}\left\Vert \phi_{\gamma}-\phi\right\Vert _{H^{s}\left(\mathbb{R}\times\mathbb{T}\right)}=0$ and $\left\Vert \phi_{\gamma}\right\Vert_{H^{s}\left(\mathbb{R}\times\mathbb{T}\right)}\leq C\left\Vert \phi\right\Vert _{H^{s}\left(\mathbb{R}\times\mathbb{T}\right)}$ and   $\left\{u_{\gamma}\right\}$ solutions of (\ref{bipe}) associated to the initial data $\left\{\phi_{\gamma}\right\}$ such that $u_{\gamma}\in C\left(\left[0,T'\right];H^{\infty}\left(\mathbb{R}\times\mathbb{T}\right)\right)$, $T'>0$  guaranteed by the local well-posedness of (\ref{bipe}) in $H^{s}(\mathbb{R}\times\mathbb{T})$ for $s>2$.  We can extend $u_{\gamma}$ on a time interval $[0,T]$,  $T=T\left(\left\Vert \phi\right\Vert _{H^{s}},s\right)$ by Proposition \ref{constante} and also we show that  there is a constant $C_{T}$  such that,
\begin{align}\label{desinfinito2}
	\int_{0}^{T}\left(\left\Vert u_{\gamma}\right\Vert _{L_{xy}^{\infty}}+\left\Vert \partial_{x} u_{\gamma}\right\Vert _{L_{xy}^{\infty}}+\left\Vert \partial_{y}u_{\gamma}\right\Vert _{L_{xy}^{\infty}}\right)dt\leq C_{T}
\end{align}
We deduce  from energy estimate Lemma \ref{supper} and previous inequality  (\ref{desinfinito2}) that,
\begin{align}\label{ct22}
	\underset{0<t<T}{\sup}\left\Vert u_{\gamma}\right\Vert _{H^{s}\left(\mathbb{R}\times\mathbb{T}\right)}\leq C_{T}
\end{align}
by using Gronwall's inequality and inequality (\ref{desinfinito2}) we get,
\begin{align}\label{ldos}
	\underset{\gamma,\mu\rightarrow\infty}{\lim}\underset{0<t<T}{\sup}\left\Vert u_{\gamma}-u_{\mu}\right\Vert _{0}=0  
\end{align}
Now, an interpolation argument together with the inequality (\ref{ct22}) and (\ref{ldos}), allow us to find  $u\;\in C\left(\left[0,T\right];H^{s^{,}}\left(\mathbb{R}\times\mathbb{T}\right)\right)\;\cap\; L^{\infty}\left(\left[0,T\right];H^{s}\left(\mathbb{R}\times\mathbb{T}\right)\right)$ with $s^{,}<s$, such that $u_{\gamma}\rightarrow u$ in $C\left(\left[0,T\right];H^{s^{,}}\left(\mathbb{R}\times\mathbb{T}\right)\right)$. By  (\ref{ldos}), $u_{\gamma}\rightarrow u$ in $ C\left(\left[0,T\right];L^{2}\left(\mathbb{R}\times\mathbb{T}\right)\right)$. Hence, by weak* compactness $u\;\in L^{\infty}\left(\left[0,T\right];H^{s}\left(\mathbb{R}\times\mathbb{T}\right)\right)$. The 
uniqueness of $u$ is obtain for similar argument as above and the continuous dependence is consequence of the  Bona-Smith's method (see\cite{bona1975initial}).

\bibliographystyle{amsplain}
\bibliography{RCMBibTeX}

\end{document}